\documentclass[12pt]{article}
\usepackage[dvips]{color, graphicx}
\usepackage{amsmath}
\usepackage{amsfonts}
\usepackage{latexsym, amssymb}
\providecommand{\U}[1]{\protect\rule{.1in}{.1in}}

\newtheorem{theorem}{Theorem}[section]

\newtheorem{corollary}[theorem]{Corollary}

\newtheorem{example}[theorem]{Example}

\newtheorem{proposition}[theorem]{Proposition}

\newenvironment{proof}[1][Proof]{\textbf{#1.} }{\ \rule{0.5em}{0.5em}}

\def\N{\mathbb{N}}

\def\Z{\mathbb{Z}}
\newcommand{\vs}[1]{\langle #1 \rangle}

\begin{document}
\title{On the number of numerical semigroups $\vs{a,b}$ of prime power genus}
\date{\today}
\author{Shalom Eliahou\footnote{eliahou@lmpa.univ-littoral.fr.}\; and Jorge Ram\'\i{}rez Alfons\'{\i}n\footnote{jramirez@math.univ-montp2.fr}}
\maketitle

\begin{abstract}
Given $g\ge 1$, the number $n(g)$ of numerical semigroups $S \subset \N$ of genus $|\N \setminus S|$ equal to $g$ is the subject of challenging conjectures of Bras-Amor\'os. In this paper, we focus on the counting function $n(g,2)$ of \textit{two-generator} numerical semigroups of genus $g$, which is known to also count certain special factorizations of $2g$. Further focusing on the case $g=p^k$ for any odd prime $p$ and $k \ge 1$, we show that $n(p^k,2)$ only depends on the class of $p$ modulo a certain explicit modulus $M(k)$. The main ingredient is a reduction of $\gcd(p^\alpha+1, 2p^\beta+1)$ to a simpler form, using the continued fraction of $\alpha/\beta$. We treat the case $k=9$ in detail and show explicitly how $n(p^9,2)$ depends on the class of $p$ mod $M(9)=3 \cdot 5 \cdot 11 \cdot 17 \cdot 43 \cdot 257$.
\medskip

\end{abstract}

\section{Introduction}

A \textit{numerical semigroup} is a subset $S \subset \N$ containing 0, stable under addition and with finite complement in $\N$. The cardinality of $\N \setminus S$ is then called the \textit{gap number} or the \textit{genus} of $S$. It is well known that, given $g \in \N$, there are only finitely many numerical semigroups of genus $g$. Yet the question of \textit{counting them} seems to be a very hard problem, analogous to the one of counting numerical semigroups by Frobenius number. See \cite{Bras-Amoros08, Bras-Amoros09} for some nice conjectures about it. The problem becomes more tractable when restricted to semigroups $S = \vs{a,b} =\N a +\N b$ with two generators. So, let us denote by $n(g,2)$ the number of numerical semigroups $S= \vs{a,b}$ of genus $g$. On the one hand, determining $n(g,2)$ is linked to hard factorization problems, like factoring Fermat and Mersenne numbers \cite{Eliahou Ramirez}. On the other hand, the value of $n(g,2)$ is known for all $g=2^k$ with $k \ge 1$, and for all $g=p^k$ with $p$ an odd prime and $k \le 8$. Indeed, exact formulas are provided in \cite{Eliahou Ramirez}, showing in particular that $n(p^k,2)$  for $k=$ 1, 2, 3, 4, 5, 6, 7 and 8 only depends on the class of $p$ modulo 3, 1, 15, 7, 255, 31, 36465 and 27559, respectively. See also Section~\ref{k <= 8}, where these formulas are given in a new form.

Our purpose in this paper is to extend our understanding of $n(p^k, 2)$ to arbitrary exponents $k \in \N$. Giving exact formulas in all cases is out of reach since, for instance, a formula for $n(p^{4097},2)$ would require the still unknown factorization of the 12th Fermat number $2^{2^{12}}+1$. However, what can and will be done here  is to show that, \textit{for all $k \ge 1$, the value of $n(p^k,2)$ only depends on the class of $p$ modulo some explicit modulus $M(k)$}. 

This result is formally stated and proved in Section~\ref{main result}. Here is how $M(k)$ is  defined:
$$
M(k)\, =\, 
\textrm{rad}(\prod\limits_{i=1}^{k} \left(2^{i/\gcd(i,k)}-(-1)^{k/\gcd(i,k)}\right)),
$$
where rad($n$) denotes the product of the distinct prime factors of $n$, i.e. the largest square-free divisor of $n$. We start by recalling in Section~\ref{factorizations} that $n(g,2)$ can be identified with the counting function of certain special factorizations of $2g$. In Section~\ref{reduction}, we reduce $\gcd(p^\alpha+1, 2p^\beta+1)$ for $\alpha,\beta \in \N$ to the simpler form $$\gcd(p^{\gcd(\alpha,\beta)} \pm 2^\rho,c)$$ where $\rho,c \in \Z$ only depend on $\alpha,\beta$ and not on $p$. This reduction uses the continued fraction of $\alpha/\beta$ and directly leads to our main result in Section~\ref{main result}. In Section~\ref{X_{a,q}}, we introduce basic binary functions $X_{a,q}$ which will serve as building blocks in our formulas. The case $k=9$ is treated in detail in Section~\ref{k=9}, where we give an explicit formula for $n(p^9,2)$ depending on the class of $p$ mod $M(9)=3 \cdot 5 \cdot 11 \cdot 17 \cdot 43 \cdot 257$. We also provide a formula in the case $k=10$ with somewhat less details. Finally, in the last section we give and prove new formulas for $n(p^k,2)$ with $k \le 8$ in terms of the $X_{a,q}$. 

Background information on numerical semigroups can be found in the books \cite{Ramirez Alfonsin, Rosales Garcia-Sanchez}.

\section{Special factorizations of $2g$}\label{factorizations}

We first recall from \cite{Eliahou Ramirez} that $n(g,2)$ \textit{can be identified with the counting number of factorizations $uv$ of $2g$ in $\N$ satisfying $\gcd(u+1,v+1)=1$}. In formula: 
\begin{equation}\label{special}
n(g,2) \,=\,  \#\{\{u,v\} \subset \N \mid uv=2g, \, \gcd(u+1,v+1) \, = \, 1\}.
\end{equation}
This follows from the classical theorem of Sylvester \cite{Sylvester} stating that whenever $\gcd(a,b)=1$, the genus $g$ of the numerical semigroup $S = \vs{a,b}$ is given by
$$
g \,=\, \frac{(a-1)(b-1)}2.
$$
For $g=p^k$ with $p$ an odd prime, an immediate consequence of (\ref{special}) is the following formula.
\begin{proposition}\label{2g} For any odd prime $p$ and exponent $k \ge 1$, we have
\begin{equation*}\label{eq p^k}
n(p^k,2) \, =\,  \#\{0 \le i \le k \mid \gcd(p^i+1,2p^{k-i}+1)=1\}. \, {\ \rule{0.5em}{0.5em}}
\end{equation*} 
\end{proposition}

Thus, in order to understand the behavior of $n(p^k,2)$, we need to gain some control on 
$$
\gcd(p^\alpha+1, 2p^\beta+1) 
$$
for $\alpha, \beta \in \N$, and hopefully find ways to determine when this greatest common divisor equals 1. This is addressed in the next section.

\section{On $\gcd(p^\alpha+1,2p^\beta+1)$}\label{reduction}
Here is the key technical tool which will lead to our main result in Section~\ref{main result}. Given $\alpha, \beta \in \N$, we shall reduce the greatest common divisor
$$
\gcd(p^\alpha+1,2p^\beta+1)
$$
to the simpler form
$$
\gcd(p^\delta\pm 2^\rho, c),
$$
where $\delta=\gcd(\alpha,\beta)$ and where $\rho, c \in \Z$ only depend on $\alpha, \beta$ and not on $p$. For this purpose, it is more convenient to work in the ring $\Z[2^{-1}]$ where 2 is made invertible. Moreover, one may effortlessly replace $\Z[2^{-1}]$ by any unique factorization domain $A$,  and 2 by any invertible element $u$ in $A$. Of course then, the gcd is only defined up to invertible elements of $A$. The proof in this more general context remains practically the same. 

\begin{proposition}\label{prop} Let $A$ be a unique factorization domain and let $x,u \in A$ with $u$ invertible. Let $\alpha, \beta \in \N$ and set  $\delta=\gcd(\alpha,\beta)$. Then there exists $\rho \in \Z$ such that
$$
\gcd(x^{\alpha}+1,u x^{\beta}+1)\,=\,\gcd(x^{\delta}\pm u^\rho, u^{\alpha/\delta} - (-1)^{(\alpha-\beta)/\delta}).
$$
\end{proposition}

The proof is based on a careful study of the successive steps in the Euclidean algorithm for computing gcd's. 

\medskip
\noindent
\begin{proof} First note that, since $u$ is invertible, we have
$$
\gcd(x^{\alpha}+1,u x^{\beta}+1)\,=\,\gcd(x^{\alpha}+1,x^{\beta}+u^{-1}).
$$
Set $r_0=\alpha,r_1=\beta$. Consider the Euclidean algorithm to compute $\gcd(r_0,r_1)$: 
\begin{equation}\label{r_i}
r_i=a_i r_{i+1}+r_{i+2}
\end{equation}
for all $0 \le i \le n-1$, where $0 \le r_{i+1} < r_i$ for all $1 \le i \le n-1$, $r_{n+1}=0$, $r_n=\gcd(r_0,r_1)$. Of course, the $a_i$'s are the \textit{partial quotients} of the continued fraction $[a_0,a_1,\ldots,a_n]$ of $\alpha/\beta$. We have
\begin{equation}\label{a_i}
\begin{pmatrix} r_i \\ r_{i+1}
\end{pmatrix}
= 
\begin{pmatrix} a_i & 1 \\ 1 & 0
\end{pmatrix}
\begin{pmatrix} r_{i+1} \\ r_{i+2}
\end{pmatrix}
\end{equation}
for all $0 \le i \le n-1$. Set $(s_0,s_1)=(1,1)$ and $(t_0,t_1)=(0,-1)$. Then we have
$$
\begin{matrix}
x^{r_0}+1 & = & x^{r_0}-(-1)^{s_0}u^{t_0},\\
x^{r_1}+u^{-1} & = & x^{r_1}-(-1)^{s_1}u^{t_1}.
\end{matrix}
$$
For $i=0, \ldots, n-1$, recursively define
\begin{eqnarray*}
s_{i+2} & = &s_i - a_i s_{i+1},\\
t_{i+2} & = &t_i - a_i t_{i+1}.
\end{eqnarray*}
Then as in \eqref{a_i}, we have 
\begin{eqnarray}
\begin{pmatrix}\label{s_i} s_i \\ s_{i+1}
\end{pmatrix}
& = & 
\begin{pmatrix} a_i & 1 \\ 1 & 0
\end{pmatrix}
\begin{pmatrix} s_{i+1} \\ s_{i+2}
\end{pmatrix}, \\
\begin{pmatrix} t_i \\ t_{i+1}
\end{pmatrix}
& = & 
\begin{pmatrix}\label{t_i} a_i & 1 \\ 1 & 0
\end{pmatrix}
\begin{pmatrix} t_{i+1} \\ t_{i+2}
\end{pmatrix}
\end{eqnarray}
for all $0 \le i \le n-1$.
Finally, for all $0 \le j \le n+1$, set
$$
f_j=x^{r_j}-(-1)^{s_j}u^{t_j}.
$$
Note that $f_0=x^{r_0}+1$, $f_1=x^{r_1}+u^{-1}$, and 
\begin{equation}\label{f_{n+1}}
f_{n+1}=1-(-1)^{s_{n+1}}u^{t_{n+1}}
\end{equation}
since $r_{n+1}=0$.

\medskip
\noindent
\textbf{Claim.} For all $0 \le i \le n-1$, we have
\begin{equation}\label{claim}
\gcd(f_{i},f_{i+1}) = \gcd(f_{i+1},f_{i+2}).
\end{equation}

\noindent
Indeed, it follows from \eqref{r_i} that
\begin{eqnarray*}
f_i & = & x^{r_i}-(-1)^{s_i}u^{t_i}\\
    & = & (x^{r_{i+1}})^{a_i}x^{r_{i+2}}-(-1)^{s_i}u^{t_i}.
\end{eqnarray*}
Now, since
$$
x^{r_{i+1}} \equiv (-1)^{s_{i+1}}u^{t_{i+1}} \bmod f_{i+1},
$$
we find
\begin{eqnarray*}
f_i & \equiv & ((-1)^{s_{i+1}}u^{t_{i+1}})^{a_i}x^{r_{i+2}}-(-1)^{s_i}u^{t_i} \bmod f_{i+1} \\
    & \equiv & (-1)^{a_i s_{i+1}}u^{a_i t_{i+1}}x^{r_{i+2}}-(-1)^{s_i}u^{t_i} \bmod f_{i+1}.
\end{eqnarray*}
Thus,
\begin{eqnarray*}
(-1)^{-a_i s_{i+1}}u^{-a_i t_{i+1}}f_i & \equiv & x^{r_{i+2}}-(-1)^{s_i-a_i s_{i+1}}u^{t_i-a_i t_{i+1}} \bmod f_{i+1}\\
& \equiv & x^{r_{i+2}}-(-1)^{s_{i+2}}u^{t_{i+2}} \bmod f_{i+1} \\
& \equiv & f_{i+2}  \bmod f_{i+1}.
\end{eqnarray*}
Consequently, we have $f_i \equiv (-1)^{a_i s_{i+1}}u^{a_i t_{i+1}}f_{i+2} \bmod f_{i+1}$. Using the equality 
$$\gcd(f,g)=\gcd(g,h)$$ 
whenever $f \equiv h \bmod g$ for elements in $A$, we conclude that
\begin{eqnarray*}
\gcd(f_i,f_{i+1}) & = & \gcd(f_{i+1},(-1)^{a_i s_{i+1}}u^{a_i t_{i+1}}f_{i+2}) \\
                  & = & \gcd(f_{i+1},f_{i+2})
\end{eqnarray*}
since $(-1)^{a_i s_{i+1}}u^{a_i t_{i+1}}$ is a unit in $A$. This proves the claim. 

\bigskip

As a first consequence, we get
\begin{equation}\label{reduc}
\gcd(f_0,f_1)=\gcd(f_n,f_{n+1}).
\end{equation}
Denote now
$$
A = \prod_{i=0}^{n-1} \begin{pmatrix} a_i & 1 \\ 1 & 0 \end{pmatrix} = \begin{pmatrix} \alpha_{11} & \alpha_{12} \\ \alpha_{21} & \alpha_{22} \end{pmatrix}.
$$
We have $\det A = (-1)^n$, and it follows from repeatedly applying \eqref{a_i} that
$$
\begin{pmatrix} r_0 \\ r_{1}
\end{pmatrix}
= 
A
\begin{pmatrix} r_{n} \\ 0
\end{pmatrix}.
$$
This implies, in particular, that $\alpha_{11}=r_0/r_n$ and $\alpha_{21}=r_1/r_n$. Similarly, using \eqref{t_i} repeatedly, we have
$$
A^{-1}
\begin{pmatrix} t_0 \\ t_{1}
\end{pmatrix} =
\begin{pmatrix} t_{n} \\ t_{n+1}
\end{pmatrix}.
$$
Since $A^{-1}=(-1)^n\begin{pmatrix} \alpha_{22} & -\alpha_{12} \\ -\alpha_{21} & \alpha_{11} \end{pmatrix}$ and 
$\begin{pmatrix} t_0 \\ t_{1} \end{pmatrix}= \begin{pmatrix} 0 \\ -1\end{pmatrix}$, this implies that 
$$
t_{n+1}=(-1)^{n+1} \alpha_{11}=(-1)^{n+1} r_0/r_n.
$$
Finally, using \eqref{s_i} repeatedly, we have
$$
A^{-1}
\begin{pmatrix} s_0 \\ s_{1}
\end{pmatrix} =
\begin{pmatrix} s_{n} \\ s_{n+1}
\end{pmatrix}.
$$
As above, and since $\begin{pmatrix} s_0 \\ s_{1} \end{pmatrix}= \begin{pmatrix} 1 \\ 1\end{pmatrix}$, we find that
$$
s_{n+1}=(-1)^n (-\alpha_{21}+\alpha_{11}) = (-1)^n (r_0-r_1)/r_n.
$$
Summarizing, it follows from the equality (\ref{reduc}), the expression (\ref{f_{n+1}}) for $f_{n+1}$, and the above values of $s_{n+1},t_{n+1}$, that 
\begin{eqnarray*}
\gcd(x^{\alpha}+1,u x^{\beta}+1) & = & \gcd(f_n,f_{n+1})\\
& = & \gcd(x^{r_n}-(-1)^{s_n}u^{t_n}, 1-(-1)^{s_{n+1}}u^{t_{n+1}}) \\
& = & \gcd(x^{\delta}-(-1)^{s_n}u^{t_n}, u^{\alpha/\delta}-(-1)^{(\alpha-\beta)/\delta}).
\end{eqnarray*}
\end{proof}

\bigskip

The special case of interest to us, namely where $A=\Z[2^{-1}]$ and $u=2$, reduces to the following statement.

\begin{corollary}\label{gcd 2} Let $1 \le i \le k$ be given integers, and set $\delta=\gcd(i,k)$. Then there exists $\rho \in \Z$ such that for any odd prime $p$, we have
$$
\gcd(p^{i}+1,2p^{k-i}+1)\,=\,\gcd(p^{\delta}\pm 2^\rho, 2^{i/\delta}-(-1)^{k/\delta}).
$$
\end{corollary}
\begin{proof} First observe that $\gcd(p^{i}+1,2p^{k-i}+1)$ is odd since the second argument is, so we may as well work in $\Z[2^{-1}]$ when computing this gcd. Set $\alpha=i$, $\beta=k-i$. Since $\gcd(i,k-i)=\gcd(i,k)$, the values of $\delta$ in Proposition~\ref{prop} and here are the same. Now $(\alpha-\beta)/\delta=(2i-k)/\delta$, and so
$$
(-1)^{(\alpha-\beta)/\delta} \,=\,(-1)^{k/\delta}.
$$
The claimed formula for $\gcd(p^{i}+1,2p^{k-i}+1)$ now follows directly from that in Proposition~\ref{prop}. 
\end{proof}

\medskip

Consequently, given $1 \le i \le k$, an odd prime $p$ satisfies the condition 
$$\gcd(p^i+1, 2p^{k-i})+1\, =\,1$$ 
if and only if $p$ belongs to a certain union of classes mod $(2^{i/\delta}-(-1)^{k/\delta})$, where as above $\delta=\gcd(i,k)$. This is the key to our main result below.

\section{The main result}\label{main result}

For a positive integer $n$, let $\textrm{rad}(n)$ denote the \textit{radical} of $n$, i.e. the product of the distinct primes factors of $n$. For instance, $\textrm{rad}(4)=2$ and $\textrm{rad}(6)=\textrm{rad}(12)=\textrm{rad}(18)=6$. Given $k \ge 1$, let us define

$$
M(k)\, =\, 
\textrm{rad}(\prod\limits_{i=1}^{k} \left(2^{i/\gcd(i,k)}-(-1)^{k/\gcd(i,k)}\right)).
$$
Note that if $k$ is odd, the formula becomes
$$
M(k)\, =\, 
\textrm{rad}(\prod\limits_{i=1}^{k} \left(2^{i/\gcd(i,k)}+1\right)),
$$
whereas if $k$ is even there is no such reduction in general,  since the exponent $k/\gcd(i,k)$ may assume both parities. Here is our main result.

\begin{theorem}\label{main} For any odd prime $p$ and $k \ge 1$, the value of $n(p^k,2)$ only depends on the class of $p$ modulo $M(k)$.
\end{theorem}
\begin{proof}
Recall the formula given by Proposition~\ref{2g}:
\begin{equation}\label{eq p^k}
n(p^k,2) = \#\{0 \le i \le k \mid \gcd(p^i+1,2p^{k-i}+1)=1\}.
\end{equation}
If $i=0$, then $\gcd(2,2p^{k}+1)=1$ always, since $p$ is odd. Assume now $1 \le i \le k$, and set $$m_k(i)=2^{i/\gcd(i,k)}-(-1)^{k/\gcd(i,k)}.$$ By Corollary~\ref{gcd 2}, \textit{the value of $\gcd(p^i+1,2p^{k-i}+1)$ only depends on the class of $p$ mod $m_k(i)$}. Therefore, it follows from (\ref{eq p^k}) and this property of $m_i(k)$ that if we set
$$
M(k)\, =\, \textrm{rad}(\prod\limits_{i=1}^{k}m_k(i)),
$$
the value of $n(p^k,2)$ only depends on the class of $p$ mod $M(k)$.
\end{proof}

\bigskip

For concreteness, Table 1 gives the value of $M(k)$ for $1 \le k \le 10$. We have seen that $n(p^k,2)$ only depends on the class of $p$ modulo $M(k)$. But $M(k)$ is not necessarily the \textit{smallest} modulus with this property, only a multiple of it. For instance, we have $M(4)=21$, but the value of $n(p^4,2)$ only depends on the class of $p$ mod 7, as stated in the Introduction. However, for all \textit{odd} $k$ in the range $1 \le k \le 9$, the modulus $M(k)$ actually turns out to be optimal for the desired property. (See \cite{Eliahou Ramirez} and Section~\ref{k <= 8}.)

\begin{table}[h]\label{M(k)}
{\arraycolsep2mm
$$
\begin{array}{|c||c|c|c|c|c|c|c|c|c|c|}
\hline
k &1 & 2 & 3 & 4 & 5 & 6 & 7 & 8 & 9 & 10\\  
\hline \hline 
M(k) &  3 & 3  & 15 & 21 & 255 & 465 & 36465 & 82677 & 30998055 &16548735\\ \hline
\end{array}
$$}
\caption{First 10 values of $M(k)$.}
\end{table}

\section{The basic functions $X_{a,q}$} \label{X_{a,q}}
We now introduce numerical functions $X_{a,q}$, with values in $\{0,1\}$, which will subsequently serve as building blocks in our explicit formulas for $n(p^k,2)$ with $k \le 10$. Given integers $a,q$ with $q \ge 2$, the definition of 
$$
X_{a,q}: \Z \to \{0,1\}
$$
depends on the distinct prime factors of $q$, as follows.
\begin{itemize}
\item If $q$ is prime, then $X_{a,q}$ is the indicator function of the complement of the subset $a+q\Z$ in $\Z$, i.e. 
$$
X_{a,q}(n)\,=\,\left\{
\begin{array}{ll}
1 & \textrm{ if } n\not\equiv a \bmod q,\\
0 & \textrm{ if } n\equiv a \bmod q.
\end{array}
\right.
$$
\item If $q_1,\ldots,q_t$ are the distinct prime factors of $q$, then we set
$$
X_{a,q} \,=\, \prod_{i=1}^t X_{a,q_i}.
$$
\end{itemize}
In particular, since $X_{a,q}$ only depends on the prime factors of $q$, we have
$$
X_{a,q} \,=\, X_{a,\textrm{rad}(q)}.
$$
Note that $X_{a,q}$ \textit{only depends on the class of $a$ mod $q$.} It is also plain that $X_{a,q}(n)$ only depends on the class of $n$ mod $q$. 

\bigskip

We now establish a few more properties of these functions. The first one links $X_{a,q}(n)$ with $\gcd(n-a,q)$, and so will be useful to capture occurrences of the equality $\gcd(p^i+1, 2p^{k-i}+1)=1$.
\begin{proposition}
\label{X_a,q}
Let $a,q$ be integers with $q \ge 2$. For all $n\in \Z$, we have
$$
X_{a,q}(n)\,=\,\left\{
\begin{array}{ll}
1 & \textrm{ if } \gcd(n-a,q) = 1,\\
0 & \textrm{ if not.}
\end{array}
\right.
$$
\end{proposition}
\begin{proof} Let $q_1,\ldots,q_t$ be the distinct prime factors of $q$. Then we have
\begin{eqnarray*}
X_{a,q}(n) = 1
& \Longleftrightarrow & X_{a,q_i}(n) = 1 \,\,\forall i \\
& \Longleftrightarrow & n \not \equiv a \bmod q_i \,\,\forall i \\
& \Longleftrightarrow & \gcd(n-a,q_i) = 1 \,\,\forall i\\
& \Longleftrightarrow & \gcd(n-a,q)  =  1.
\end{eqnarray*}
Since $X_{a,q}(n)$ only takes values in $\{0,1\}$, this implies that $X_{a,q}(n)=0$ if and only if $\gcd(n-a,q)  \not=  1$.
\end{proof}

\bigskip

Next, for determining $n(p^k,2)$, we often need to evaluate $X_{a,q}(p^s)$ with $s \ge 2$. The next two properties help remove that exponent $s$. The first one reduces the task to the case where $s$ divides $q-1$. It suffices to consider the case where $q$ is prime.

\begin{proposition} Let $q$ be a prime number, and let $a,s$ be integers with $s\ge 2$. Write $s=t e$ with $t =\gcd(s,q-1)$, so that $\gcd(e,q-1)=1$. Let $d \in \N$ satisfy $d e \equiv 1 \bmod q-1$. Then
$$
X_{a,q}(n^s) \,=\, X_{a^d,q}(n^t) 
$$
for all integers $n$.
\end{proposition}
\begin{proof} 
This is the heart of the RSA cryptographic protocol, which relies on the fact that exponentiation to the power $e$ in $\Z/q\Z$ is a bijection, whose inverse is exponentiation to the power $d$. We have
\begin{eqnarray*}
X_{a,q}(n^s) = 0
& \Longleftrightarrow & n^s \equiv a \bmod q \\
& \Longleftrightarrow & (n^t)^e \equiv a \bmod q \\
& \Longleftrightarrow & (n^t)^{de} \equiv a^d \bmod q \\
& \Longleftrightarrow & n^t \equiv a^d \bmod q \\
& \Longleftrightarrow & X_{a^d,q}(n^t) = 0.
\end{eqnarray*}
\end{proof}

\bigskip

Thus, we may now assume that the exponent $s$ divides $q-1$. 

\begin{proposition}\label{p^s} Let $q$ be a prime number, and let $a,s$ be integers with $s$ dividing $q-1$. Let $g \in \N$ be an integer whose class mod $q$ generates the multiplicative group of non-zero elements in $\Z/q\Z$. We have:
\begin{itemize}
\item If $a$ is not an $s$-power mod $q$, then $X_{a,q}(n^s) = 1$ for all $n$.
\item If $a$ in an $s$-power mod $q$, then $a \equiv g^{si} \bmod q$ for some integer $i$ such that $0 \le i \le (q-1)/s-1$, and 
$$
X_{a,q}(n^s) \,=\, \prod_{j=0}^{s-1}X_{g^{i+j(q-1)/s},q}(n)
$$
for all integers $n$.
\end{itemize}
\end{proposition}
\begin{proof} In the group $(\Z/q\Z)^*$ of nonzero classes mod $q$, the set of $s$-powers is of cardinality $(q-1)/s$ and coincides with
$$
\{g^{si} \bmod q \mid 0 \le i \le (q-1)/s-1\}.
$$
First, if $a$ is not an $s$-power mod $q$, then $n^s \not \equiv a \bmod q$ for all $n$, implying $X_{a,q}(n^s) = 1$ for all $n$. Assume now $a$ is an $s$-power mod $q$. By the above remark,  there exists $0 \le i \le (q-1)/s-1$ such that $a \equiv g^{si} \bmod q$. We have
\begin{eqnarray*}
X_{a,q}(n^s) = 0
& \iff & n^s \equiv a \bmod q \\
& \iff & n^s \equiv g^{si} \bmod q \\
& \iff & \left(\frac{n}{g^i}\right)^s \equiv 1 \bmod q.
\end{eqnarray*}
This means that $n/g^i$ is of order dividing $s$ in the group $(\Z/q\Z)^*$. Now, the elements of order dividing $s$ in this group constitute a subgroup of order $s$ generated by $g^{(q-1)/s}$. Thus, there exists an integer $j$ such that $0 \le j \le s-1$ and satisfying
$$
\frac{n}{g^i} \equiv g^{j(q-1)/s} \bmod q,
$$
yielding
$$
X_{a,q}(n^s) = 0 \iff n \equiv g^{i+j(q-1)/s} \bmod q. 
$$
Summarizing, for $a \equiv g^{si} \bmod q$, we have established the equivalence
$$
X_{a,q}(n^s) = 0 \iff \prod_{j=0}^{s-1}X_{g^{i+j(q-1)/s},q}(n) = 0,
$$
whence the claimed equality $X_{a,q}(n^s) = \prod_{j=0}^{s-1}X_{g^{i+j(q-1)/s},q}(n)$.
\end{proof}

\begin{example}\label{X_8,17} In order to establish our formula for $n(p^{10},2)$ in Section~\ref{k=9}, the term $X_{8,17}(p^2)$ turns out to be involved. Now $8$ is a square mod $17$, namely $8 \equiv 5^2 \equiv 12^2 \bmod 17$. Thus, the above result yields 
$$X_{8,17}(p^2)\,=\,X_{5,17}(p)X_{12,17}(p).$$ 
\end{example}

\section{The cases $k=9, 10$}\label{k=9}
Explicit formulas for $n(p^k,2)$ with $p$ an odd prime and $k \le 6$ or $k=8$ are given in \cite{Eliahou Ramirez}. Here we go further and treat the case $k=9$ in detail. This will show how Corollary~\ref{gcd 2} can be applied, and will also give a sense of the increasing complexity of these formulas. We also briefly address the case $k=10$. The main ingredients are the basic functions $X_{a,q}$ defined in the preceding section.

Here comes our formula for $n(p^9,2)$. The fact that it depends on the class of $p$ mod $M(9)$ follows from this prime decomposition: 
$$M(9)\,=\,30998055\,=\,5 \cdot 17 \cdot 257 \cdot 3 \cdot 11 \cdot 43.$$

\begin{theorem} Let $p$ be an odd prime. Then we have
$$
n(p^9,2)=1+2X_{3,5}(p)+X_{9,17}(p)+X_{128,257}(p)+X_{2,3}(p)\cdot(3+X_{2,11}(p)+X_{8,43}(p)).
$$
\end{theorem}

\medskip
\noindent
\begin{proof} By Proposition~\ref{2g}, in order to determine $n(p^9,2)$, it suffices to count those exponents $i$ between 0 and 9 satisfying $\gcd(p^i+1, 2p^{9-i}+1)=1$. Using Corollary~\ref{gcd 2} and the calculations leading to it, these gcd's may be reduced as follows:
\begin{eqnarray*}
\gcd(p^0+1, 2p^9+1) & = & 1\\
\gcd(p^1+1, 2p^8+1) & = & \gcd(p+1,3)\\
\gcd(p^2+1, 2p^7+1) & = & \gcd(2p-1,5)\\
\gcd(p^3+1, 2p^6+1) & = & \gcd(p^3+1,3)\,=\,\gcd(p+1,3)\\
\gcd(p^4+1, 2p^5+1) & = & \gcd(2p-1,17)\\
\gcd(p^5+1, 2p^4+1) & = & \gcd(p-2,33)\\
\gcd(p^6+1, 2p^3+1) & = & \gcd(2p^3+1,5)\\
\gcd(p^7+1, 2p^2+1) & = & \gcd(p-8,129)\\
\gcd(p^8+1, 2p^1+1) & = & \gcd(2p+1,257)\\
\gcd(p^9+1, 2p^0+1) & = & \gcd(p^9+1,3)\,=\,\gcd(p+1,3).
\end{eqnarray*}
Now, by Proposition~\ref{X_a,q} and the properties of the functions $X_{a,q}$, these equalities imply the following equivalences:
\begin{eqnarray*}
\gcd(p^0+1, 2p^9+1) = 1 & & \textrm{always}\\
\gcd(p^1+1, 2p^8+1) = 1 & \iff & X_{2,3}(p) =1\\
\gcd(p^2+1, 2p^7+1) = 1 & \iff & X_{3,5}(p) =1\\
\gcd(p^3+1, 2p^6+1) = 1 & \iff & X_{2,3}(p) =1\\
\gcd(p^4+1, 2p^5+1) = 1 & \iff & X_{9,17}(p) =1\\
\gcd(p^5+1, 2p^4+1) = 1 & \iff & X_{2,33}(p) =1\\
\gcd(p^6+1, 2p^3+1) = 1 & \iff & X_{3,5}(p) =1\\
\gcd(p^7+1, 2p^2+1) = 1 & \iff & X_{8,129}(p) =1\\
\gcd(p^8+1, 2p^1+1) = 1 & \iff & X_{128,257}(p) =1\\
\gcd(p^9+1, 2p^0+1) = 1 & \iff & X_{2,3}(p) =1.\\
\end{eqnarray*}
Read sequentially, this table directly yields the following first formula for $n(p^9,2)$, with 10 summands, in terms of the functions $X_{a,q}$:
\begin{eqnarray*}
n(p^9,2) & = & 1+X_{2,3}(p)+X_{3,5}(p)+X_{2,3}(p)+X_{9,17}(p) 
+X_{2,33}(p)\\ & & \quad +X_{3,5}(p)+X_{8,129}(p)+X_{128,257}(p)+X_{2,3}(p) \\
& = & 1+3X_{2,3}(p)+2X_{3,5}(p)+X_{9,17}(p) +X_{2,33}(p)+X_{8,129}(p)\\ & & \quad +X_{128,257}(p).
\end{eqnarray*}
Among the moduli involved above, the only non-prime ones are $33 = 3 \cdot 11$ and $129=3 \cdot 43$. By definition of $X_{a,q}$ for non-prime $q$, we have
\begin{eqnarray*}
X_{2,33} & =\ & X_{2,3}X_{2,11}\\
X_{8,129} & = & X_{8,3}X_{8,43}.
\end{eqnarray*}
Moreover, since $X_{a,q}$ only depends on the class of $a$ mod $q$, we have $$X_{8,3}=X_{2,3}.$$ 
Substituting these equalities in the above formula for $n(p^9,2)$, we get
$$
n(p^9,2)=1+2X_{3,5}(p)+X_{9,17}(p)+X_{128,257}(p)+X_{2,3}(p)\cdot(3+X_{2,11}(p)+X_{8,43}(p)),
$$
as claimed.
\end{proof}

\bigskip
\medskip

We now derive another version of our formula for $n(p^9,2)$, from which its values are easier to read. Given positive integers $q_1,\ldots, q_t$, we denote by
$$
\rho_{q_1,\ldots,q_t}: \Z \; \rightarrow\; \Z/q_1 \Z \times \cdots \times \Z/q_t\Z
$$
the canonical reduction morphism $\rho_{q_1,\ldots,q_t}(n) = (n \bmod q_1, \ldots, n \bmod q_t)$. Moreover, we write $n \equiv \neg a \bmod q$ instead of $n \not\equiv a \bmod q$. For example, the condition $$\rho_{5,17,257}(p) = (3,\neg 9,\neg 128)$$ means $p \equiv 3 \bmod 5$, $p \not\equiv 9 \bmod 17$ and $p  \not\equiv 128 \bmod 257$.

\medskip

\begin{corollary} Let $p$ be an odd prime. Consider the following functions of $p$ depending on its classes mod $5, 17, 257$ and $11, 43$, respectively:
\begin{eqnarray*}
\lambda(p) & = & \left\{
\begin{array}{rcl}
1 & \textrm{ if } & \rho_{5,17,257}(p) = (3, 9,128) \\
2 & \textrm{ if } & \rho_{5,17,257}(p) \in \{(3, 9, \neg 128), (3, \neg  9, 128)\} \\
3 & \textrm{ if } & \rho_{5,17,257}(p) \in \{(3, \neg 9, \neg 128), (\neg 3, 9, 128)\} \\
4 & \textrm{ if } & \rho_{5,17,257}(p) \in \{(\neg 3, 9, \neg 128), (\neg 3, \neg 9, 128)\} \\
5 & \textrm{ if } & \rho_{5,17,257}(p) = (\neg 3, \neg 9, \neg 128),
\end{array}
\right.\\ \\
\mu(p) & = & \left\{
\begin{array}{rcl}
3 & \textrm{ if } & \rho_{11,43}(p) = (2, 8) \\
4 & \textrm{ if } & \rho_{11,43}(p) \in \{(2, \neg 8), (\neg 2, 8)\} \\
5 & \textrm{ if } & \rho_{11,43}(p) = (\neg 2, \neg 8). \\
\end{array}
\right.
\end{eqnarray*}
Then we have
$$
n(p^9,2) \,=\, \left\{
\begin{array}{ll}
\lambda(p) & \textrm{ if } p\equiv 2 \bmod 3,\\
\lambda(p) + \mu(p) & \textrm{ if } p\not\equiv 2 \bmod 3.
\end{array}
\right.
$$
\end{corollary}
\begin{proof} This directly follows from the preceding result and the easy to prove equalities
\begin{eqnarray*}
\lambda(p) & = & 1+2X_{3,5}(p)+X_{9,17}(p)+X_{128,257}(p),\\
\mu(p) & = & 3+X_{2,11}(p)+X_{8,43}(p).
\end{eqnarray*}
\end{proof}

\bigskip
It is still clearer now that $n(p^9,2)$ is determined by the class of $p$ mod $M(9)=3\cdot 5\cdot 17\cdot 257\cdot 11\cdot 43$, and that $M(9)$ is the smallest modulus with this property.

\bigskip
We close this section by briefly treating the case $k=10$. The formula obtained shows that $n(p^{10},2)$, for $p$ an odd prime, is determined by the class of $p$ modulo $M(10)/15 \,=\, 7 \cdot 17 \cdot 73 \cdot 127$.

\begin{theorem} Let $p$ be an odd prime. Then we have
$$
n(p^{10},2)=7+X_{3,7}(p)(1+X_{36,73}(p))+X_{5,17}(p)X_{12,17}(p)+X_{123,127}(p).
$$
\end{theorem}
\begin{proof} After reducing $\gcd(p^i+1, 2p^{10-i}+1)$ for $0 \le i \le 10$ as in Corollary~\ref{gcd 2}, and using Proposition~\ref{X_a,q} involving the functions $X_{a,q}$, we obtain this first raw formula:
\begin{eqnarray*}
n(p^{10},2) & = & 2+X_{-1,3}(p^2)+X_{3,7}(p)+X_{-2,5}(p^2)+1+X_{2,9}(p^2)+X_{123,127}(p)\\& & +\, X_{8,17}(p^2)+X_{255,511}(p)+X_{-1,3}(p^{10}).
\end{eqnarray*}
We now invoke Proposition~\ref{p^s} several times. Since $-1$ is not a square mod 3, we have $X_{-1,3}(p^2)=1$. The same reason yields $X_{2,9}(p^2)=X_{-1,3}(p^{10})=1$. Similarly, we have $X_{-2,5}(p^2)=1$ as $-2$ is not a square mod 5. As already explained in Example~\ref{X_8,17}, we have $X_{8,17}(p^2)=X_{5,17}(p)X_{12,17}(p)$. Finally, since $511=7\cdot 73$, and since 255 is congruent to 3 mod 7 and to 36 mod 73, we have
$$
X_{255,511}(p) \,=\, X_{3,7}(p)X_{36,73}(p).
$$
Inserting these reductions into the raw formula gives the stated one, where now the only argument of the various basic functions $X_{a,q}$ is $p$ and all involved $q$'s are primes.
\end{proof}

\section{The cases $k \le 8$ revisited} \label{k <= 8}

While explicit formulas for $n(p^k,2)$ with $k \le 6$ and $k=8$ are given in \cite{Eliahou Ramirez}, we provide here new, shorter formulas in terms of the basic functions $X_{a,q}$ for $k \le 8$, including $k=7$. The construction method is similar to the cases $k=9,10$ and relies on the reduction of $\gcd(p^i+1, 2p^{k-i}+1)$ provided by Corollary~\ref{gcd 2}.
\begin{theorem} Let $p$ be an odd prime. Then we have
\begin{eqnarray*}
n(p^1,2)& =& 1+X_{2,3}(p)\\
n(p^2,2)& =& 3\\
n(p^3,2)& =& 1+2X_{2,3}(p)+X_{2,5}(p)\\
n(p^4,2)& =& 4+X_{3,7}(p)\\
n(p^5,2)& =& 1+3X_{2,3}(p)+X_{3,5}(p)+X_{8,17}(p)\\
n(p^6,2)& =& 6+X_{15,31}(p)\\
n(p^7,2)& =& 1+X_{2,3}(p)(3+X_{7,11}(p))+X_{2,5}(p)(1+X_{6,13}(p))+X_{2,17}(p)\\
n(p^8,2)& =& 6+X_{5,7}(p)+X_{23,31}(p)+X_{63,127}(p).
\end{eqnarray*}
\end{theorem}
\begin{proof}
Corollary~\ref{gcd 2} and its proof method yield the following reductions of $\gcd(p^i+1, 2p^{k-i}+1)$ for $i=1, \ldots, k$. The case $i=0$ is omitted, as $\gcd(p^0+1, 2p^{k}+1)=1$ always. A few more arithmetical reductions are also applied. For instance, the equality $\gcd(p^2+1,3)\,=\,1$ below follows from the fact that $-1$ is not a square mod 3. This is one easy case of Proposition~\ref{p^s}.
{\small
\begin{eqnarray*}
k & = & 1:\\
\gcd(p^1+1, 2p^0+1) & = & \gcd(p+1,3)\\ \\
k & = & 2:\\
\gcd(p^1+1, 2p^1+1) & = & \gcd(2p+1,1)\,=\,1\\
\gcd(p^2+1, 2p^0+1) & = & \gcd(p^2+1,3)\,=\,1\\ \\
k & = & 3:\\
\gcd(p^1+1, 2p^2+1) & = & \gcd(p+1,3)\\
\gcd(p^2+1, 2p^1+1) & = & \gcd(2p+1,5)\\
\gcd(p^3+1, 2p^0+1) & = & \gcd(p^3+1,3)\,=\,\gcd(p+1,3)\\ \\
k & = & 4:\\
\gcd(p^1+1, 2p^3+1) & = & \gcd(p+1,1)\,=\,1\\
\gcd(p^2+1, 2p^2+1) & = & \gcd(2p^2+1,1)\,=\,1\\
\gcd(p^3+1, 2p^1+1) & = & \gcd(2p+1,7)\\
\gcd(p^4+1, 2p^0+1) & = & \gcd(p^4+1,3)\,=\,1
\end{eqnarray*}
\begin{eqnarray*}
k & = & 5:\\
\gcd(p^1+1, 2p^4+1) & = & \gcd(p+1,3)\\
\gcd(p^2+1, 2p^3+1) & = & \gcd(2p-1,5)\\
\gcd(p^3+1, 2p^2+1) & = & \gcd(p-2,9)\\
\gcd(p^4+1, 2p^1+1) & = & \gcd(2p+1,17)\\
\gcd(p^5+1, 2p^0+1) & = & \gcd(p^5+1,3)\,=\,\gcd(p+1,3)\\ \\
k & = & 6:\\
\gcd(p^1+1, 2p^5+1) & = & \gcd(p+1,1)\,=\,1\\
\gcd(p^2+1, 2p^4+1) & = & \gcd(p^2+1,3)\,=\,1\\
\gcd(p^3+1, 2p^3+1) & = & \gcd(2p^3+1,1)\,=\,1\\
\gcd(p^4+1, 2p^2+1) & = & \gcd(2p^2+1,5)\,=\,1\\
\gcd(p^5+1, 2p^1+1) & = & \gcd(2p+1,31)\\
\gcd(p^6+1, 2p^0+1) & = & \gcd(p^6+1,3)\,=\,1\\ \\
k & = & 7:\\
\gcd(p^1+1, 2p^6+1) & = & \gcd(p+1,3)\\
\gcd(p^2+1, 2p^5+1) & = & \gcd(2p+1,5)\\
\gcd(p^3+1, 2p^4+1) & = & \gcd(2p-1,9)\\
\gcd(p^4+1, 2p^3+1) & = & \gcd(p-2,17)\\
\gcd(p^5+1, 2p^2+1) & = & \gcd(p+4,33)\\
\gcd(p^6+1, 2p^1+1) & = & \gcd(2p+1,65)\\
\gcd(p^7+1, 2p^0+1) & = & \gcd(p^7+1,3)\,=\,\gcd(p+1,3)\\ \\
k & = & 8:\\
\gcd(p^1+1, 2p^7+1) & = & \gcd(p+1,1)\,=\,1\\
\gcd(p^2+1, 2p^6+1) & = & \gcd(p^2+1,1)\,=\,1\\
\gcd(p^3+1, 2p^5+1) & = & \gcd(p+2,7)\\
\gcd(p^4+1, 2p^4+1) & = & \gcd(2p^4+1,1)\,=\,1\\
\gcd(p^5+1, 2p^3+1) & = & \gcd(4p+1,31)\\
\gcd(p^6+1, 2p^2+1) & = & \gcd(2p^2+1,7)\,=\,1\\
\gcd(p^7+1, 2p^1+1) & = & \gcd(2p+1,127)\\
\gcd(p^8+1, 2p^0+1) & = & \gcd(p^8+1,3)\,=\,1.
\end{eqnarray*}
}

As in the case $k=9$, the claimed formulas follow by reading these tables sequentially and using properties of the functions $X_{a,q}$ from Section~\ref{X_{a,q}}. 
\end{proof}

\bigskip

In particular, these formulas confirm that for $k=1, \ldots, 8$, the value of $n(p^k, 2)$ at an odd prime $p$ is determined by the class of $p$ modulo 3, 1, $3\cdot 5$, 7, $3\cdot 5\cdot 17$, 31, $3\cdot 5\cdot 11\cdot 13\cdot 17$ and $7\cdot 31\cdot 127$, respectively.

\section{A question}

We shall conclude this paper with an open question. On the one hand, we have obtained explicit formulas for $n(p^k,2)$ in all cases $k \le 10$. On the other hand, we know from \cite{Eliahou Ramirez} that no such formula can be expected in the case $k=4097$, at least as long as the prime factors of the 12th Fermat number $2^{2^{12}}+1$ remain unknown. Well then, what happens in the intermediate range $11 \le k \le 4096$? Are there fundamental obstacles which would prevent us to obtain exact formulas for $n(p^k,2)$ all the way up to $k = 4096$?

\bigskip

\noindent
{\small
\textbf{Authors addresses:}

\bigskip
\smallskip

\noindent
$\bullet$ Shalom Eliahou\textsuperscript{a,b,c},

\noindent
\textsuperscript{a}Univ Lille Nord de France, F-59000 Lille, France\\
\textsuperscript{b}ULCO, LMPA J.~Liouville, B.P. 699, F-62228 Calais, France\\
\textsuperscript{c}CNRS, FR 2956, France

\bigskip

\noindent
$\bullet$ Jorge Ram\'\i{}rez Alfons\'{\i}n,

\noindent
Institut de Math\'ematiques et de Mod\'elisation de Montpellier\\
Universit\'e Montpellier 2\\
Case Courrier 051\\
Place Eug\`ene Bataillon\\
34095 Montpellier, France\\
UMR 5149 CNRS
}


\begin{thebibliography}{99}
\bibitem{Bras-Amoros08} \textsc{M.~Bras-Amor\'os}, Fibonacci-like behavior of the number of numerical semigroups of a given genus,  Semigroup Forum  76  (2008) 379--384.
\bibitem{Bras-Amoros09} \textsc{M.~Bras-Amor\'os}, Bounds on the number of numerical semigroups of a given genus, J. Pure and Applied Algebra 213 (2009) 997--1001.
\bibitem{Eliahou Ramirez} \textsc{S.~Eliahou and J.L.~Ram\'\i{}rez Alfons\'\i{}n}, Two-generator numerical semigroups and Fermat and Mersenne numbers, SIAM J. Discrete Math. 25 (2011) 622--630.
\bibitem{Ramirez Alfonsin} \textsc{J.L.~Ram\'\i{}rez Alfons\'\i{}n}, The Diophantine Frobenius problem. Oxford Lecture Series in Mathematics and its Applications 30, Oxford University Press, Oxford, 2005.
\bibitem{Rosales Garcia-Sanchez}  \textsc{J.C.~Rosales and P.A.~Garc\'\i{}a-S\'anchez},  Numerical semigroups. Developments in Mathematics, 20. Springer, New York, 2009. \bibitem{Sylvester} \textsc{J.J.~Sylvester}, On subinvariants, i.e. semi-invariants to binary quantities of an unlimited order, Amer. J. Math. 5 (1882) 119--136.
\end{thebibliography}
\end{document}